\documentclass{article}
\usepackage{mathtools,amssymb,amsthm}
\usepackage{tikz}
\usetikzlibrary{arrows,shapes,positioning,calc}
\usepackage{tkz-graph,tkz-berge}
\tikzset{VertexStyle/.style ={shape= circle, fill = white, inner sep = 0pt, outer sep = 0pt, minimum size = 5pt,draw}}
\newtheorem{theorem}{Theorem}[section]
\newtheorem{corollary}{Corollary}[section]
\newtheorem{lemma}{Lemma}[section]
\newtheorem{proposition}{Proposition}[section]
\newtheorem{conjecture}{Conjecture}[section]
\theoremstyle{definition}
\newtheorem{definition}{Definition}[section]
\newtheorem{example}{Example}[section]
\newtheorem{remark}{Remark}[section]

\newtheorem{question}{Question}[section]

\newcommand*{\vphi}{\varphi}%varphi en plus court
\newcommand*{\defi}[1]{\emph{#1}}%pour définition, permet de changer la mise en page
\newcommand*{\tn}[1]{\textnormal{#1}} %\textnormal
\newcommand*{\defegal}{\coloneqq}%définition
\newcommand*{\iso}{\simeq}%isomorphisme,commande sémantique (\simeq ou \cong)
\newcommand*{\corps}[1]{\mathbf{#1}}%corps, commande sémantique
\newcommand*{\Z}{\corps{Z}}
\newcommand*{\ZZ}{\Z/2\Z}
\newcommand*{\tiret}{\nobreakdash-\hspace{0pt}} %pas de césure dans $math$-blabla, mais permet la césure du mot qui suit.
\newcommand*{\enstq}[2]{\{#1 \mid #2\}}% ensemble 1, tel que 2
\newcommand*{\presentation}[2]{\langle#1\mid#2\rangle}
\newcommand*{\restreinta}{\mathclose{}|\mathopen{}}%restriction
\DeclareMathOperator{\id}{Id}
\DeclareMathOperator{\SL}{SL}
\DeclareMathOperator{\diam}{diam}
\DeclareMathOperator{\schrei}{Sch}
\DeclareMathOperator{\cayl}{Cay}
\DeclareMathOperator{\Aut}{Aut}

\DeclareMathOperator*{\bigast}{\ast}

\DeclareMathOperator{\Star}{Star}

\DeclareMathOperator{\rank}{rank}
\DeclarePairedDelimiter\abs{\lvert}{\rvert}

\newcommand*{\factor}[1]{$#1$\tiret factor}
\newcommand*{\group}{\mathcal}%semantic command
\newcommand*{\G}{{\group G}}
\newcommand*{\h}{{\group H}}
\newcommand*{\Cayley}[2]{\cayl(\group #1,#2^\pm)}
\newcommand*{\Schrei}[3]{\schrei(\group #1,\group #2,#3^\pm)}
\newcommand*\No{No.}
\newcommand*\auteur{Paul-Henry Leemann}
\author{\auteur}
\title{Schreier graphs: transitivity and coverings.}
\begin{document}
\maketitle
\begin{abstract}
We give a characterization of isomorphisms between Schreier graphs in terms of the groups, subgroups and generating systems.
This characterization may be thought as a graph analog of Mostow's rigidity theorem for hyperbolic manifolds.
This allows us to give a transitivity criterion for Schreier graphs.
Finally, we show that Tarski monsters satisfy a strong simplicity criterion.
\end{abstract}
\section{Introduction}
It is well-known that if $\group H$ is a subgroup of a group $\G=\langle X\rangle$, $\group H$ is normal if and only if the corresponding Schreier graph $\Schrei{G}{H}{X}$ is vertex-transitive by automorphisms preserving the labeling.
It is also known that if $\Schrei{G}{H_1}{X}$ and $\Schrei{G}{H_2}{X}$ are isomorphic, then the subgroups $H_1$ and $H_2$ are isomorphic, but the converse is not true.

In this paper, we give a characterization of isomorphisms between Schreier graphs in terms of the groups, subgroups and generating systems.
In the case of regular graphs of even degree, this characterization may be thought of as a rigidity result ``a la Mostow''.
As a corollary, we have a characterization of vertex-transitive Schreier graphs (by automorphisms that may not preserve the labeling) in terms of the subgroup $H$.
Such subgroups will be called length-transitive.
They generalize the notion of normal subgroups.
This leads to a strengthening of the notion of simple group.
We prove that this notion is not equivalent to simplicity, by showing that for odd $n\geq 5$ alternating groups $\group A_n$ are not strongly simple in this sense.
We also exhibit non-trivial examples of strongly simple groups, namely Tarski groups.
These infinite strongly simple groups also allow us to partially answer a question of Benjamini concerning coverings of graphs in which the cover is a Cayley graph.

The paper is organized as follows.
Sections \ref{SectNotations} and \ref{sectionSchreier} introduce all the relevant notions and useful preliminary results.
In Section \ref{sectionTransitivity}, we prove our main theorem (Theorem \ref{KeyTheorem}) on isomorphisms between Schreier graphs and some corollaries on transitivity.
We also give a reformulation (Theorem \ref{ThmRigidity}) of our main theorem to make the relation with Mostow's rigidity theorem  more apparent.
In the next section, we investigate coverings and label-preserving coverings (also called $X$-coverings) of Schreier graphs.
Finally, in Section \ref{sectionGroups}, we define a stronger notion of simplicity for groups and prove that this definition is not equivalent to simplicity.
We use this to show that the Cayley graph of a Tarski monster can not $X$-cover an infinite transitive graph (distinct from itself).
\section{Notations and Definitions}\label{SectNotations}
For us, a \defi{graph} $\Gamma$ consists of two sets $E$ (\defi{edges}) and $V$ (\defi{vertices}), and two functions $\iota\colon E\to V$ and $\bar{\phantom{e}}\colon E\to E$ satisfying $\bar{\bar{ e}}=e$.
The vertex $\iota(e)$ is the initial vertex and the vertex $\tau(e)\defegal\iota(\bar e)$ is the final vertex of the edge $e$.
The edge $\bar e$ is the inverse of the edge $e$.
An unoriented edge is a pair $\{e,\bar e\}$.
The \defi{degree} of a vertex is the number of outgoing edges (equivalently the number of incomming edges).
A graph is \defi{locally finite} if every vertex has finite degree.
We will say that an edge $e$ is \defi{degenerate} if $\bar e=e$.
Note that this is possible only if $e$ is a loop.
Remark that a vertex with a unique loop has degree $1$ if the loop is degenerate and $2$ otherwise.
A graph with no degenerate loops correspond to the definition of a graph by Serre \cite{Serre} and many results about such graphs from \cite{Serre} or \cite{Stalling} can be extended easily to the general case.

For a set $X$ with an involution $\phantom x^{-1}\colon X\to X$, a \defi{labeling} of $\Gamma$ by $X$ consists of a function $f\colon E\to X$ such that $f(\bar e)=f(e)^{-1}$.
A \defi{morphism of graphs} is a map $\phi\colon \Gamma_1\to \Gamma_2$ wich preserves the graph structure,
meaning that $\phi(\bar e)=\overline{\phi(e)}$ and $\phi(\iota(e))=\iota(\phi(e))$.
If $\Gamma_1$ and $\Gamma_2$ are labeled graphs over the same set $X$ with label functions $f_1$ and $f_2$, we say that $\phi$ is an \defi{$X$-morphism} (or \defi{morphism of labeled graphs}) if $\phi$ is a morphism of graphs such that $f_1=f_2\cdot\phi$.
The set of all isomorphisms from a graph $\Gamma$ to itself is denoted by $\Aut(\Gamma)$.

The geometric realization of labeled graphs on figures is the following.
Vertices of the graph are drawn as nodes (fat points) and unoriented labeled edges $\{e,\bar e\}$ as labeled curves that join them.
If $e$ and $\bar e$ have same label $a$ (i.e. if $a=a^{-1}$), the corresponding curve is undirected and labeled by $a$.
If $e$ and $\bar e$ have labels $x$ and $x^{-1}\neq x$ respectively, the corresponding curve is directed from $\iota(e)$ to $\tau(e)$ and labeled by $x$.
See Figure \ref{ExempleProp} for a example of such a geometric realization.

It is easy to see that for every rooted labeled graph $\Gamma$ such that for each vertex $v$ and each label $a$ there exists at most one edge with initial vertex $v$ and label $a$, the only $X$-automorphism of $\Gamma$ that sends the root to the root is the identity.

A graph $\Gamma$ is said to be \defi{vertex-transitive} (or simply \defi{transitive}) if for every pair of vertices $x$ and $y$, there exists an automorphism $\phi\colon\Gamma\to\Gamma$ with $\phi(x)=y$.
The graph $\Gamma$ is \defi{almost transitive} if there exists a finite set $V_0$ of vertices such that every vertex of $\Gamma$ can be mapped onto $V_0$ by an automorphism of $\Gamma$.
A labeled graph is \defi{(almost) $X$-transitive} if it is (almost) transitive by $X$-automorphisms.

A \emph{generating system} $X$ of a group $\group A$ is a multiset of elements of $\group A$ ---~i.e. $X$ contains only elements of $\group A$ and an element $x\in X$ may appear more than once~--- such that the group $A$ is generated by elements of $X$.
The \defi{Cayley graph} of $\group A$ with respect to $X$ is the labeled graph $\Cayley{A}{X}$ with vertex set $\group A$, and for every $x$ with $x\in X$ or $x^{-1}\in  X$, an edge from $g$ to $h$ labeled by $x$ if and only if $h=gx$. 
Note that with this definition, loops and multiple edges (two vertices are connected by at least two edges) are allowed.

For $X$ and $\group A$ as before and $\h$ a subgroup of $\group A$, the \defi{Schreier graph} of $\h$ in $\group A$ (with respect to $X$) is defined as the graph $\Schrei{A}{H}{X}$ with vertices the right cosets $Hg=\enstq{hg}{h\in H}$ and with an edge labeled by $x$ from $Hg_1$ to $Hg_2$ if and only if $Hg_2=Hg_1x$.

Note that Cayley and Schreier graphs are rooted graphs in the sense that they have a distinguished vertex: $1$ for Cayley graphs and $\h$ for  Schreier graphs.

A (possibly non-labeled) graph $\Gamma$ is said to be a \defi{Cayley graph} (respectively a \defi{Schreier graph}) if it is isomorphic (as a non-labeled, non-rooted graph) to some Cayley graph (resp. Schreier graph).
With these definitions, it is easy to see that if $\h$ is a normal subgroup of $\group A$, then $\Schrei{A}{H}{X}\iso\Cayley{A/H}{X}$ is a Cayley graph.
This result justifies the particular definition of a Cayley graph that we use (allowing loops and multiples edges).

Let $\group A$ be a group with generating system $X$.
For any $g\in\group A$, the \defi{length of $g$ with respect to $X$} is $\abs g_X\defegal \min\enstq{n\in \corps N}{g=x_1\dots x_n,\ x_i\in X}$.
A word $w=w_1w_2\dots w_n$ on the alphabet $X$ is \defi{reduced} if $\abs w_X=n$.

\section{Basic facts about Schreier graphs}\label{sectionSchreier}
All Cayley and Schreier graphs are connected by definition.
Thus, from now on and unless otherwise specified, we will assume all graphs in the paper to be connected.

It is well-known that a graph $\Gamma$ is a Cayley graph of a group $\group A$ if and only if there exists a free and transitive action of $\group A$ on $\Gamma$.
Moreover, given a graph $\Gamma$ with a simply transitive action of a group $\group A$, Sabidussi shows in \cite{Sabidussi} an explicit way to put labels on edges of $\Gamma$ so as to make it a Cayley graph of $\group A$.
Namely, choose any vertex $v_0$ as the root and for any neighbor $w_i$ of $v_0$, label the edge from $v_0$ to $w_i$ by the unique element $x_i$ of $\group A$ that sends $v_0$ on $w_i$.
Then, use $x_i$ to label the remaining edges.
For example, the edge from $w_i$ to some vertex $u$ is labeled by $x_j$ if and only if $x_jx_i$ sends $v_0$ to $u$.
It is then easy to check that the action of $\group A$ is (in fact) also $X$-transitive.

A graph $\Gamma$ is a Schreier graph of some group if and only if it is a Schreier graph of a free product of the form
\begin{align}
\G=(\bigast_I \Z) \ast(\bigast_J \ZZ)=\presentation{x_i,y_j,\ i\in I,\ j\in J}{y_j^2},\tag{$\star$}\label{group}
\end{align}
with generating system $X=\{x_i\}_{i\in I}\sqcup \{y_j\}_{j\in J}$.
Note that the generating system $X$ of $\G$ described above is an actual subset of $\G$.
For any group $\group A$ and generating system $Z$, there exists a group $\G$ with generating set $X$ as above and a normal subgroup $\group N$ such that $\group A\iso \G/\group N$, $Z$ is the disjoint union of the $\pi(x)$ for $x\in X$, where $\pi\colon\G\to\G/\group N$ is the natural projection, and $\pi(x)^2=1$ %in $\G/\group N$
 if and only if $x^2=1$.
The last condition ensures that $e$ and $\bar e$ are distinct in $\Schrei{G}{H}{X}$ if and only if they are distinct in $\Schrei{A}{H/N}{Z}$.
 
In fact, a graph is a Schreier graph if and only if it admits a decomposition into disjoint $1$ and \factor{2}s, where an \defi{\factor{n}} of a graph $\Gamma$ is a subgraph $\Delta$ of $\Gamma$ such that every vertex of $\Gamma$ has degree $n$ in $\Delta$.
Here, the \factor1s correspond to subgraphs consisting of edges labeled by a generator of order $2$ and the \factor{2}s to subgraphs with edges labeled by a generator of infinite order in the group $\G$.

This fact can been used to show that every regular graph of even degree without degenerate loops is a Schreier graph over a free group (\cite{Gross} for the finite case, and \cite{Cannizzo} for the locally finite case).
On the other side, Godsil and Royl showed that every finite transitive graph of odd degree admits a \factor{1}, see \cite{GodsilRoyle}.
This result extends to locally finite infinite transitive graphs of odd degree, using compacity and results from Aharoni (\cite{Aharoni}) on matchings in infinite graphs.
Putting all this together, we have that every locally finite transitive graph (of odd or even degree) is a Schreier graphs over a group $\G$ of the form \eqref{group}.

From now on, the letter $\G$ will always denote a group of the form \eqref{group}. 
In such a group, the only cancellations that can occur are of the form $ww^{-1}$ where $w$ is one of the generators.

For such a group $\G$, we have the easy but useful following lemma.
\begin{lemma}\label{lemmaPathWord}
 Let $\h$ be any subgroup of $\G$ and $\Gamma\defegal\Schrei{G}{H}{X}$ be the corresponding Schreier graph.
 Then, for every vertex $v$ in $\Gamma$, there is a bijection between reduced paths starting at $v$ and elements of $\G$.
 This bijection restricts to a bijection between closed reduced paths starting at $v=\h g$ and elements of $g^{-1}\h g$.
\end{lemma}
\begin{proof}
For the case of free groups, see proposition 1.3 in \cite{NagnKaim}.
For the general case, notice that the presentation of $\G$ is chosen such that a word $w$ is reduced if and only if it doesn't contain a subword of the form $xx^{-1}$ or of the form $x^{-1}x$, where $x$ is any generator, and a path in $\Gamma$ is reduced if and only if it does not contains a subpath of the form $e\bar e$ or $\bar ee$.
\end{proof}
\section{A criterion for transitivity}\label{sectionTransitivity}

\begin{definition}\label{DefDegree}
Let $\group A$ be any group with generating system $X$.
The \defi{degree} of $\group A$ (with respect to $X$) is the degree of any vertex in the graph $\Cayley{A}{X}$.
\end{definition}
Notice that the degree of $\group A$ depends on the choice of the generating system and could be infinite.
Note that the degree of $\group A$ with respect to $X$ is in fact the sum of the number of $x\in X$ of order at most $2$ and of twice the number of $x\in X$ or order at least $3$.
\begin{definition}
 Let $\h_1=\presentation{X_1}{\mathcal R_1}$ and $\h_2=\presentation{X_2}{\mathcal R_2}$ be two arbitrary groups.
 A morphism $\alpha\colon \h_1\to\h_2$ \defi{preserves lengths} (with respect to $X_1$ and $X_2$) if for every $h$ in $\h_1$ we have $\abs{\alpha(h)}_{X_2}=\abs{h}_{X_1}$.
 If $\alpha$ is an isomorphism, we say that $\h_1$ and $\h_2$ are \defi{length-isomorphic}.
\end{definition}
\begin{definition}\label{LengthTrans}
Let $\G$ be as in  \eqref{group}.
A subgroup $\h$ of $\G$ is \defi{length-transitive} if it is length-isomorphic to all its conjugates.
That is, if for every $g$ in $\G$, there exists a group isomorphism $\alpha_g\colon\mathcal  H\to g^{-1}\h g$ which preserves lengths.
\end{definition}
\begin{remark}
In this definition, $\alpha_g$ is only defined on $\h$, not on $\G$ itself.

In general, we have $\h\iso g^{-1}\h g$, but the conjugation homomorphism does not preserve lengths unless $\G=\Z$ or $\G=\Z/2\Z$.
\end{remark}

We are now able to state our main result.
\begin{theorem}\label{KeyTheorem}
 Let $\G_1$ and $\G_2$ be as in  \eqref{group}.
 Suppose that $\Gamma_i\defegal\Schrei{G_i}{H_i}{X_i}$ for $i=1,2$.
 Then, there exists a graph isomorphism from $\Gamma_1$ to $\Gamma_2$ that respects roots (the image of the vertex $\h_1$ is the vertex $\h_2$) if and only if $\G_1$ and $\G_2$ have same degree and $\h_1$ and $\h_2$ are length-isomorphic.
 
 Moreover, there exists an $X$-isomorphism from $\Gamma_1$ to $\Gamma_2$ that respects roots if and only if $\G_1=\G_2$ and $\h_1=\h_2$.
\end{theorem}
It is possible to reformulate this theorem in order to have a rigidity theorem ``a la Mostow''.
Let $\Gamma_1$ and $\Gamma_2$ be two $2d$\tiret regular graphs without degenerate loop.
If they are isomorphic, then their fondamental groups $\pi_1(\Gamma_1)$ and $\pi_1(\Gamma_2)$ are isomorphic as abstract groups, but the converse is not necessarily true.
On the other hand, since the graphs are $2d$\tiret regular without degenerate loop, we have two coverings $p_i\colon \Gamma_i\to R_d$, where $R_d$ is the unique graph with one vertex and $d$ loops --- see section \ref{SectionCoverings} for more on coverings.
These two coverings induce two injections ${p_i}_*\colon\pi_1(\Gamma_i)\to\pi_1(R_d)=\langle X\rangle$, where loops of $R_d$ are in bijection with elements of $X$.
The situation for odd regular graphs is more complex since the projections $p_i$ may not exist.
More precisely, if $\Gamma$ is a $2d+1$\tiret regular graph without degenerate loop and $R_{d,1}$ denotes the graph with one vertex, $d$ loops and $1$ degenerate loop, then there exists a covering $p\colon \Gamma\to R_{d,1}$ if and only if $\Gamma$ admits a perfect matching, if and only if $\Gamma$ is isomorphic to a Schreier graph.
This gives us the alternative formulation (for regular graphs) of Theorem \ref{KeyTheorem}.
\begin{theorem}[Rigidity theorem for regular graphs]\label{ThmRigidity}
Let $\Gamma_1$ and $\Gamma_2$ be two locally finite regular graphs without degenerate loop.

If $\Gamma_1$ and $\Gamma_2$ are $2d$\tiret regular,
then they are isomorphic as graphs if and only if ${p_1}_*\bigl(\pi_1(\Gamma_1)\bigr)$ is length-isomorphic to ${p_2}_*\bigl(\pi_1(\Gamma_2)\bigr)$.

If $\Gamma_1$ and $\Gamma_2$ are $2d+1$\tiret regular and both admit a perfect matching,
then they are isomorphic as graphs if and only if ${p_1}_*\bigl(\pi_1(\Gamma_1)\bigr)$ is length-isomorphic to ${p_2}_*\bigl(\pi_1(\Gamma_2)\bigr)$.

Moreover, these statements are independent of the choice of the coverings $p_1$ and $p_2$.
\end{theorem}

To prove Theorem \ref{KeyTheorem}, we need to prove two implications.
The first one is easy and proven in the next proposition. % \ref{PropTransitive}.
The second one is a little harder and is the subject of Proposition \ref{PropConverse}.
\begin{proposition}\label{PropTransitive}
Let $\G_i$, $\h_i$, $X_i$ and $\Gamma_i$ be as in Theorem \ref{KeyTheorem}, for $i=1,2$.
Recall that the graph $\Gamma_i$ is naturally a rooted graph, with root (the vertex) $\h_i$.
Suppose that there exists an isomorphism $\beta\colon\Gamma_1\to\Gamma_2$ such that $\beta(\h_1)=\h_2$, then $\h_1$ and $\h_2$ are length-isomorphic and $\G_1$ and $\G_2$ have same degree.

Moreover, if $\beta$ preserves labeling, then $\G_1=\G_2$ and $\h_1=\h_2$.
\end{proposition}
\begin{proof}
 The graphs $\Gamma_1$ and $\Gamma_2$ being isomorphic, they have same degree.
 Hence, the groups $\G_1$ and $\G_2$ have same degree too.

  Now, pick an element $h$ of $\h_1$.
  By Lemma \ref{lemmaPathWord}, $h$ is represented by a closed path $p$ with base-point $\h_1$.
  If we apply $\beta$ we have a closed path $\beta(p)$ with base-point $\h_2$.
  Define a map from $\h_1$ to $\G_2$ by 
  \[\alpha(h)\defegal\textnormal{label in } \Gamma_2 \textnormal{ of }\beta(p).\]
  Since $\beta$ induces a bijection between closed paths with base-point $\h_1$ and closed paths with base-point $\h_2$, $\alpha$ is a bijection between $\h_1$ and $\h_2$.
  Moreover, $\alpha(1)=1$ and $\alpha(h^{-1})=\alpha(h)^{-1}$ (the path is read backward).
  We also have $\alpha(h_1h_2)=\alpha(h_1)\alpha(h_2)$ (the paths are read one after the other).
  This proves that $\alpha$ is a group isomorphism between $\h_1$ and $\h_2$.
  The fact that $\alpha$ preserves lengths is trivial.
  
Now, suppose that $\beta$ preserves labelings.
In this case, we immediately have $\G_1=\G_2$ and $\alpha=\id$.
\end{proof}

\begin{example}
The Petersen graph is $3$\tiret regular and hence can be seen has a Schreier graph: $\Gamma=\Schrei{\Z*\Z/2\Z}{H}{X}$, see Figure \ref{ExempleProp}.
It is a well-known fact that it is transitive, but not a Cayley graph.
Now, let us denote by $\h$ (resp. $\group M$) the group of labels of closed reduced paths based at $v_1$ (resp. $w_1$) in Figure \ref{ExempleProp}.
We have $\group M=a\h a$.
The element $xax^{-2}a$ belongs to $\h$ but not to $\group M$, therefore $\h$ and $\group M$ are not equal and both are not normal.
This means that there exists no $X$-automorphism of $\Gamma$ sending $v_1$ to $w_1$.
But there exists an automorphism $\beta$ that does the job.
And therefore there exists an isomorphism $\alpha\colon \h\to\group M$ that preserves lengths.
We want to compute $\alpha(xax^{-2}a)$.
The automorphism $\beta$ is given by:
\begin{align*}
v_1&\mapsto w_1 & w_1&\mapsto v_1\\
v_2&\mapsto w_3 & w_2&\mapsto v_3\\
v_3&\mapsto w_5 & w_3&\mapsto v_5\\
v_4&\mapsto w_2 & w_4&\mapsto v_2\\
v_5&\mapsto w_4 & w_5&\mapsto v_4.
\end{align*}
For two adjacent edges $a$ and $b$ in $\Gamma$, let us denote the unique edge from $a$ to $b$ by $[ab]$.
Then $xax^{-2}a\in\h$ corresponds to the path $[v_1v_3][v_3w_3][w_3w_2][w_2w_1][w_1v_1]$.
This path is sent by $\beta$ to the path $[w_1w_5][w_5v_5][v_5v_3][v_3v_1][v_11_1]$, which has label $x^{-1}ax^{-2}a\in\group M$.
Therefore, $\alpha(xax^{-2}a)=x^{-1}ax^{-2}a$.

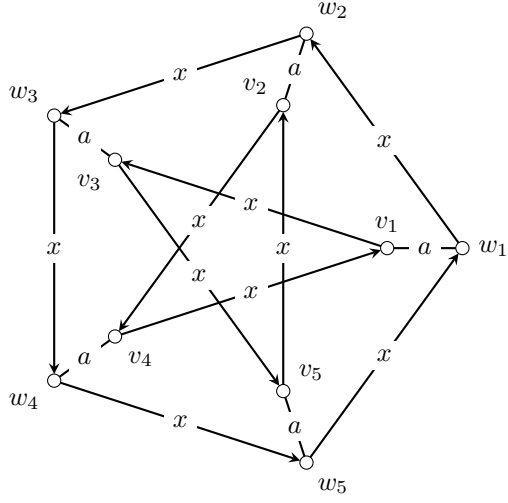
\begin{figure}[!h]
\centering
\begin{tikzpicture}[>=stealth]
\SetVertexMath
\Vertices[LabelOut,Lpos=90,unit=2]{circle}{v_1,v_2,v_3,v_4,v_5}
\Vertices[LabelOut,unit=3]{circle}{w_1,w_2,w_3,w_4,w_5}
\Edges[label=$x$,style={->}](w_1,w_2,w_3,w_4,w_5,w_1)
\Edges[label=$x$,style={->}](v_1,v_3,v_5,v_2,v_4,v_1)
\foreach \n in {1,...,5}{
	\Edge[label=$a$](v_\n)(w_\n)
}
\end{tikzpicture}
\caption{The Petersen graph viewed as a Schreier graph on $\presentation{x,a}{a^2}\iso\Z*\Z/2\Z$.}
\label{ExempleProp}
\end{figure}
\end{example}

We are going to prove the converse of Proposition \ref{PropTransitive}.
Namely, that if $\h_1$ and $\h_2$ are length-isomorphic by an isomorphism $\alpha$, then there exists an isomorphism between their Schreier graphs that preserves roots.
For that, we first extend $\alpha$ to a bijection (not a group homomorphism) from $\G_1$ to $\G_2$ and see that it is possible to find such an extension with good properties.
Then we will use such an extension to find an isomorphism $\beta$ from $\Gamma_1$ to $\Gamma_2$ such that $\beta(\h_1)=\h_2$ (as vertices).

\begin{lemma}\label{LemmaExtension}
 Let $\G_1$ and $\G_2$ be two groups with the same degree, and $\h_i$ be any subgroup of $\G_i$ for $i=1,2$.
 Then, every isomorphism $\alpha\colon \h_1\to \h_2$ which preserves lengths can be extended to a bijection $\gamma\colon \G_1\to \G_2$ such that $\gamma$ preserves lengths and initial segments.
 That is: for every $f$, $g\in \G_1$ such that $fg$ is reduced
 (i.e. $\abs{fg}_{X_1}=\abs{f}_{X_1}+\abs{g}_{X_1}$),
 there exists $w$ such that $\gamma(fg)=\gamma(f)w$ with $\abs w_{X_2}=\abs{\gamma(g)}_{X_2}$.
\end{lemma}
\begin{proof}
 Clearly, $\alpha$ preserves lengths and initial segments if we restrict it to $f,g\in \h_1 $.
 So let $\gamma\restreinta_{\h_1}\defegal\alpha$.
 We are now going to look at the set of initial segments of $\h_1$:
   \[
   C\defegal \enstq{f\in \G}{\exists w\in \G: fw\in \h_1 \tn{ and } fw \tn{ is reduced}}.
   \]
Let $c\in C$ be an initial segment of length $n$ of $h\in \h_1$.
Define $\gamma(c)$ as the initial segment of length $n$ of $\gamma(h)=\alpha(h)$.
We need to check that $\gamma(c)$ is well-defined.
Firstly, $cw$ is reduced and $h$ and $\gamma(h)$ are of length at least $n$, so it is possible to choose an initial segment of length $n$ of $\gamma(h)$.
Secondly, we need to check that $\gamma(c)$ does not depend on the particular choice of $h$.
Let $h_1$ and $h_2$ be elements of $\h_1$ and let $c$ be their maximal common initial segment. Then, if $c$ is of length $n$:
\begin{align*}
 \abs{h_1}_1+\abs{h_2}_1-2n&=\abs{h_1^{-1}h_2}_1& c \tn{ is maximal}\\
 &=\abs{\alpha(h_1^{-1}h_2)}_2&\alpha \tn{ preserves lengths}\\
 &=\abs{\alpha(h_1)^{-1}\alpha(h_2)}_2& \alpha \tn{ is a homomorphism}\\
 &=\abs{\alpha(h_1)^{-1}}_2+\abs{\alpha(h_2)}_2-2n'& \\
 &=\abs{h_1}_1+\abs{h_2}_1-2n',
\end{align*}
where $n'$ is the length of the maximal initial segment common to $\alpha(h_1)$ and $\alpha(h_2)$, and $\abs \cdot_i$ is short for $\abs \cdot_{X_i}$.
So $n=n'$, hence $\gamma(c)$ does not depend on the choice of $h=cw$.
Moreover, it is trivial that for $c_1,c_2\in C$, if $c_1$ is an initial segment of $c_2$, then $\gamma(c_1)$ is an initial segment of $\gamma(c_2)$.
We have thus a bijection between $C$ and $\gamma(C)$ which preserves lengths and initial segments.
The groups $\G_1$ and $\G_2$ having same degree, they are length-isomorphic.
This induces a bijection which preserves lengths:
  \[
  \gamma'\colon D\defegal \G_1\setminus C\longrightarrow \G_2\setminus \gamma(C).
  \]
We now need to define $\gamma$ on $D$.
The set $C$ being closed under the operation ``initial segment'', no elements of $D$ are initial segments of elements of $C$.
We can thus define $\gamma$ on $D$ from the ``bottom''.
Let $D_n$ be the set of elements of $D$ of length $n$ and let $n_0$ be the smallest integer such that $D_{n_0}$ is non-empty --- it is also the smallest integer such that $\gamma'(D_{n_0})$ is non-empty.
Let $d\in D_{n_0}$.
By minimality of $n_0$, there exists $c\in C$ and $x\in X_1^\pm$ such that $d=cx$ is reduced.
Similarly, for $d'\in \gamma'(D_{n_0})$ there exists $c'\in C$ and $y\in X_2^\pm$ such that $d'=\gamma'(c')y$ is reduced.

We are now going to prove that the following two sets are in bijection:
 \begin{align*}
   E&\defegal\enstq{x\in X_1^\pm}{cx\in D_{n_0}}\\
   F&\defegal\enstq{y\in X_2^\pm}{\gamma(c)y\in \gamma'(D_{n_0})}.
 \end{align*}
To show that, we are going to prove that their complements $\bar E\subset {X_1}^\pm$ and $\bar F\subset {X_2}^\pm$ are in bijection.
 These complements are exactly 
 \begin{align*}
   \bar E&=\enstq{x\in X_1^\pm}{cx\in C}\\
   \bar F&=\enstq{y\in X_2^\pm}{\gamma(c)y\in \gamma(C)}.
 \end{align*}
For $x\in \bar E$, we have $\gamma(cx)=\gamma(c)y$ for a unique $y\in X_2^\pm$.
This defines a map $\theta\colon \bar E \to \bar F$ by $\theta(x)=y$.
This map is injective because $\theta(x)=\theta(x')$ if and only if $\gamma(cx)=\gamma(cx')$ and so if and only if $x=x'$.
On the other hand, $\theta$ is also surjective.
Indeed, if $y$ is in $\bar F$, then $\gamma(c)y$ is an element of $\gamma(C)$.
Hence, there exists $c'\in C$ such that $\gamma(c')=\gamma(c)y$.
But $\gamma$ preserves initial segments on $C$, thus $c$ is an initial segment of $c'$, hence $c'=cx$ for some $x$.
This finishes the proof of the existence of a bijection between $\bar E$ and $\bar F$ and therefore between $E$ and $F$.

This bijection allows us to define $\gamma(d)=\gamma(cx)\defegal\gamma(c)y$ for $d\in D_{n_0}$.
We have thus extended $\gamma$ to $C\cup D_{n_0}$ such that $\gamma$ is a bijection which preserves lengths and initial segments.
Finally, we put $C_0\defegal C$ and conclude by induction on $C_i\defegal C_{i-1}\cup D_{n_i}$ and $D_{n_i}$ where $n_i$ is the smallest integer greater than $n_{i-1}$ such that $D_{n_i}$ is non-empty.
\end{proof}
\begin{lemma}\label{LemmaProperties}
The bijection $\gamma\colon \G_1\to\G_2$ of the preceding lemma has the following properties:
\begin{enumerate}
 \item For all $fg^{-1}\in \h_1$ reduced, $\gamma(fg^{-1})=\gamma(f)\gamma(g)^{-1}$;
 \item $\gamma^{-1}$ preserves lengths and initial segments;
 \item For all $fg^{-1}\in \gamma(\h_1)=\h_2$ reduced, $\gamma^{-1}(fg^{-1})=\gamma^{-1}(f)\gamma^{-1}(g)^{-1}$.
\end{enumerate}
\end{lemma}
\begin{proof}
 If $fg^{-1}$ is reduced and is an element of $\h_1$, the same is true for its inverse $gf^{-1}$.
 But then, there exists $w$ and $w'$ in $\h_2$ such that $\gamma(fg^{-1})=\gamma(f)w$ and $\gamma(gf^{-1})=\gamma(g)w'$ are reduced.
 The bijection $\gamma$ being a group homomorphism on $\h_1$, we have:
    \[
      e=\gamma(e)=\gamma(fg^{-1}gf^{-1})=\gamma(fg^{-1})\gamma(gf^{-1})=\gamma(f)w\cdot\gamma(g)w'.
    \]
The only reductions possible are between $w$ and $\gamma(g)$, which are of the same length.
Hence $w=\gamma(g)^{-1}$, which is what we wanted to prove.

For the second part, it is trivial that $\gamma^{-1}$ preserves lengths.
For the initial segments part, let $fg\in\h_2$ be reduced.
Then $\gamma^{-1}(fg)=f'g'$ is reduced, where $\abs{f'}_1=\abs{\gamma^{-1}(f)}_1=\abs{f}_2$ (and analogously for $\abs{g'}_1$).
If we apply $\gamma$ to both sides of the equality, we have $fg=\gamma(f'g')=\gamma(f')w$ reduced, for some $w\in\h_2$.
We conclude that $f=\gamma(f')$.
Therefore, $\gamma^{-1}(fg)=\gamma^{-1}(f)g'$ is reduced.

The last point can be proved in the same way as for $\gamma$, using the fact that the restriction of $\gamma^{-1}$ to $\gamma(\h_1)$ is a group homomorphism.
\end{proof}

\begin{proposition}\label{PropConverse}
Let $\G_i$, $\h_i$, $X_i$ and $\Gamma_i$ be as in Theorem \ref{KeyTheorem}.
If $\G_1$ and $\G_2$ have same degree and $\h_1$ and $\h_2$ are length-isomorphic, then there exists a graph isomorphism $\beta\colon\Gamma_1\to\Gamma_2$ that respects roots
(i.e. $\beta$ maps the vertex $\h_1$ to the vertex $\h_2$).
\end{proposition}
\begin{proof}
 Let $\alpha \colon \h_1\to \h_2$ be an isomorphism which preserves lengths.
 We extend $\alpha$ to $\gamma \colon \G_1\to \G_2$ as in Lemma \ref{LemmaExtension}.
 We define $\beta$ on vertices by $\beta(\h_1 f)\defegal \h_2\gamma(f)$.
 It is trivial that $\beta(\h_1)=\h_2$.
 Moreover, $\beta$ is well-defined and injective.
 Indeed, $\h_1 f=\h_1 g$ if and only if $fg^{-1}$ is an element of $\h_1$.
 In the same way, $\h_2 \gamma(f)=\h_2 \gamma(g)$ if and only if $\gamma(f)\gamma(g)^{-1}=\gamma(fg^{-1})$ is an element of $\h_2=\gamma(\h_1)$.
 Hence, $\h_1 f=\h_1 g$ if and only if $\h_2 \gamma(f)=\h_2 \gamma(g)$.
 Finally, for every vertex $\h_2 h$ of $\Gamma_2$ ($h$ an element of $\G_2$), $\h_2 h=\h_2 \gamma(\gamma^{-1}(h))$ with $\gamma^{-1}(h)\in\G_1$.
 Hence, $\h_2 h=\beta(\h_1 \gamma^{-1}(h))$ and $\beta$ is surjective on vertices.
 
 Instead of describing $\beta$ explicitly on edges, we are going to show that for every pair of vertices $\h_1 f$ and $\h_1 g$, the edges between $\h_1 f$ and $\h_1 g$ are in bijection with the edges between $\beta(\h_1f)$ and $\beta(\h_1 g)$.
 Taking this bijection as a definition of $\beta$ on edges makes $\beta$ an isomorphism from $\Gamma_1$ to $\Gamma_2$.
 Firstly, suppose that $\h_1 f$ and $\h_1 g$ are joined by at least one edge, labeled by $x_0\in X_1^\pm$, such that $\h_1fx_0=\h_1g$.
 Then the set of all edges from $\h_1 f$ to $\h_1g$ is 
   \[
     A\defegal\enstq{x\in X_1^\pm}{\h_1 fx=\h_1 fx_0}=\enstq{x\in X_1^\pm}{fx_0x^{-1}f^{-1}\in \h_1}.
   \]
 On the other hand, we have $\beta(\h_1 fx_0)=\h_2 \gamma(fx_0)=\h_2 \gamma(f)y_0=\beta(\h_1 f)y_0$ for a unique $y_0\in X_2^\pm$.
 Thus, there is at least one edge from $\beta(\h_1 f)$ to $\beta(\h_1 g)$, labeled by $y_0$.
 The set of all edges from $\beta(\h_1 f)$ to $\beta(\h_1 g)$ is 
   \begin{align*}
     B&\defegal
    \enstq{y\in X_2^\pm}{\gamma(f)y_0y^{-1}\gamma(f)^{-1}\in \h_2}\\
     &\, \hspace{0.5pt} =\enstq{y\in X_2^\pm}{\gamma(fx_0)y^{-1}\gamma(f)^{-1}\in \h_2}.
   \end{align*}
 Take any $x$ in $A$.
  By Lemma \ref{LemmaProperties} we have $\gamma(fx_0x^{-1}f^{-1})=\gamma(fx_0x^{-1})\gamma(f)^{-1}$ and that there exists a unique $y\in X_2^{\pm}$ such that
$\gamma(fx_0x^{-1})=\gamma(fx_0)y^{-1}$.
 Moreover, this particular $y$ belongs to $B$, thus we have a map from $A$ to $B$ and we need to show that this map is bijective.
 The map is injective.
 Indeed, if $y=y'$, then $\gamma(fx_0x^{-1})=\gamma(fx_0x'^{-1})$ and ($\gamma$ is a bijection) thus $x=x'$.
 For the surjectivity, we know that for every $y\in X_2^\pm$ there is a unique $x\in X_1^{\pm}$ such that $\gamma(fx_0x^{-1})=\gamma(fx_0)y^{-1}$, so we only need to show that if $y$ belongs to $B$ then $x$ belongs to $A$.
 If $y$ is in $B$, we have that $\gamma(fx_0)y^{-1}\gamma(f)^{-1}$ belongs to $\h_2$.
 Then by Lemma \ref{LemmaProperties} we have $\gamma(fx_0)y^{-1}\gamma(f)^{-1}=\gamma(fx_0x^{-1}f^{-1})$.
 This implies that $fx_0x^{-1}f^{-1}$ belongs to $\h_1$ and finally that $x$ is in $A$.
 
 We now need to show that if $\h_1 f$ and $\h_1 g$ are not connected by any edge, then neither are $\beta(\h_1 f)$ and  $\beta(\h_1 g)$.
 But the same argument as before shows that if $\beta(\h_1 f)$ and  $\beta(\h_1 g)$ are connected by at least one edge, then $\h_1f$ and $\h_1g$ are connected by an edge.
 
 This concludes the existence of a bijection between edges from $\beta(\h_1 f)$ to  $\beta(\h_1 g)$ and edges from $\h_1 f$ to $\h_1 g$.
 Since this bijection preserves initial and terminal vertices, we can take it as the definition of $\beta$ on edges.
 Defining $\beta$ in such a way makes it an isomorphism from $\Gamma_1$ to $\Gamma_2$ that sends $\h_1$ on $\h_2$.
 
 Now, if $\G_1=\G_2$ and $\h_1=\h_2$, the existence of an $X$-automorphism between the two Schreier graphs is trivial.
 \end{proof}
 
 This finishes the prove of Theorem \ref{KeyTheorem}.
 
 Here are two easy applications of this theorem.
 \begin{corollary}\label{CorTransitivity}
 Suppose that $\Gamma\defegal\Schrei{G}{H}{X}$.
 Then the graph $\Gamma$ is transitive if and only if the subgroup $\h$ is length-transitive.
\end{corollary}
\begin{proof}
 The graph $\Gamma$ is transitive if and only if for all $g\in \G$ there exists an automorphism of the graph that sends the vertex $\h$ to the vertex $\h g$.
 But the graph $\Gamma$ rooted in $\h g$ is exactly the graph $\schrei(\G,g^{-1}\h g,X^\pm)$.
 \end{proof}
 \begin{corollary}
 For $\group A$ a group, $X$ a generating system and $K$ a subgroup and $\Gamma$ the corresponding Schreier graph, the number of $X$-orbits is $[\group A:N_{\group A}(\group K)]$, each $X$-orbit has $[N_{\group A}(\group K):\group K]$ elements and each orbit is a union of $X$-orbits.
 \end{corollary}
 \begin{proof}
 We have $\group A=\G/\group N$ for some normal subgroup of $\G$ and $\group K$ corresponds to a $\group N\leq\h\leq\G$.
 Two vertices $\h f$ and $\h g$ are in the same $X$-orbit if and only if $f^{-1}\h f=g^{-1}\h g$.
 Therefore, the number of vertices in one orbit is $[N_{\group G}(\group H):\group H]=[N_{\group A}(\group K):\group K]$ and the number of orbits is $[\group A:N_{\group A}(\group K)]$.
Since an $X$-automorphism is an automorphism of the graph, the orbits are unions of $X$-orbits.
\end{proof}
 
 The following lemma comes to simplify the application of the transitivity criterion of Corollary \ref{CorTransitivity}.
 
\begin{lemma}
 Let $\G$, $\h$ and $X$ be as before.
 Then $\h$ is length-transitive if and only if for every $x\in X^\pm=X\cup X^{-1}$ the group $\h$ is length-isomorphic to $x^{-1}\h x$.
\end{lemma}
\begin{proof}
 The proofs of Theorem \ref{KeyTheorem} and Corollary \ref{CorTransitivity} show in fact that for every $g$ in $\G$, the existence of a length-preserving isomorphism $\alpha_g\colon \h\to g^{-1}\h g$ is equivalent to the existence of an automorphism of $\Gamma=\Schrei{G}{H}{X}$ that sends the vertex $\h$ to the vertex $\h g$.
 On the other hand, $\alpha_x$ exists for every $x\in X^\pm$ if and only if it is possible to send the vertex $\h$ to each of its neighbors by an automorphism of $\Gamma$.
Since $\Gamma$ is connected, this last condition is equivalent to the transitivity of $\Gamma$, and hence to the length-transitivity of $\h$.
\end{proof}
\begin{remark}
It is important to notice that, in order to ensure the length-transitivity of $\h$, we need to check the existence of $\alpha_x$ for all generators $x\in X$ and their inverses.
For example, the following graph is non-transitive, even if $\alpha_a$ and $\alpha_x$ exist.
Indeed, $\alpha_{x^{-1}}$ does not exist.
\begin{figure}[!h]
\centering
\begin{tikzpicture}[scale=0.7,>=stealth]
  \SetVertexNoLabel
  \SetGraphUnit{2}
  \tikzset{LabelStyle/.style ={fill=white}}
  \foreach \k in {0,1}{
  \Vertices[x=\k*3*\GraphUnit,y=0]{line}{A\k,B\k,C\k,D\k}
  \AddVertexColor{black}{B0}
  \Edge[label=$x$,style={->}](A\k)(B\k)
  \Edge[label=$x$,style={->,bend right}](B\k)(C\k)
  \Edge[label=$x$,style={->}](C\k)(D\k)
  \Edge[label=$a$,style={bend left}](B\k)(C\k)
  \Loop[dist=1.5cm,dir=NO,style={thick}, label=$a$](D\k)
  }
  \Loop[dist=1.5cm,dir=NO,style={thick}, label=$a$](A0)
  \node(0) at (-\GraphUnit*0.5,0){};
  \Edge[style={dotted}](0)(A0).
  \node(F) at (\GraphUnit*6.5,0){};
  \Edge[style={dotted}](F)(D1)
\end{tikzpicture}
\caption{A non-transitive Schreier graph over $\presentation{a,x}{a^2}$.
 The root is marked in black.
}
\label{GraphePresqueTransitif}
\end{figure}
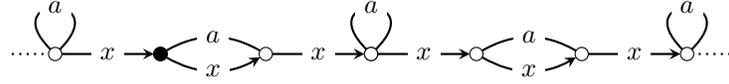
\end{remark}

 Observe that if $\Gamma$ is $X$-transitive, then $\h$ is normal and therefore $\Gamma$ is a Cayley graph.
 However, the converse does not hold.
 More precisely, let $\Gamma\defegal\Schrei{G}{H}{X}$ be a Schreier graph that is isomorphic to a Cayley graph.
 Then it is in general not true that $\Gamma$ is $X$-transitive (and that $\h$ is normal).
 All we can say is the following, which characterizes Cayley graphs among Schreier graphs. 
\begin{theorem}
 Let $\Gamma=\Schrei{G}{H}{X}$ be any Schreier graph over a group $\G$.
 Then $\Gamma$ is (isomorphic to) a Cayley graph if and only if there exists a group $\G_1=\langle X_1\rangle$ that has the same degree as $\G$ and a normal subgroup $\group N\unlhd\G_1$ which is length-isomorphic to $\h$.
 \end{theorem}
\begin{proof}
 If there exists such $\G_1$ and $\group N$, then the graph $\Gamma$ is isomorphic to the graph $\Schrei{G_1}{N}{X_1}$ which is a Cayley graph.
 
 On the other hand, if $\Gamma$ is isomorphic to a Cayley graph $\Gamma_1$, then $\Gamma_1$ is a Schreier graph $\Schrei{G_1}{N}{X_1}$ over some group $\G_1$ which has the same degree as $\G$, and for some normal subgroup $\group N$.
 Moreover, the isomorphism between $\Gamma$ and $\Gamma_1$ implies that $\h$ is length-isomorphic to a conjugate of $\group N$.
\end{proof}

\section{Coverings}\label{SectionCoverings}
In this section, we give a criterion for the existence of coverings and of $X$-coverings of Schreier graphs.
We also give some relations between $X$-coverings and quasi-isometries.

\begin{definition}
Let $\Gamma=(E,V)$ be a graph and $v$ any vertex of $\Gamma$.
The \defi{star} of $v$ is the set $\enstq{e\in E}{\iota(e)=v}$.
\end{definition}
Observe that any morphism $\vphi\colon\Gamma_1\to \Gamma_2$ induces, for any vertex $v$ of $\Gamma_1$, a~map:
\[
\vphi_v\colon\Star_v\to\Star_{\vphi(v)}.
\]
\begin{definition}
A morphism $\vphi\colon\Gamma_1\to\Gamma_2$ is a \defi{covering} if all induced maps $\vphi_v$ are bijections.
For a vertex $v$ of $\Gamma_2$, the \defi{fiber} over $v$ is the set of all preimages of $v$ by $\vphi$.

If $\Gamma_1$ is a labeled graph, $\vphi$ is \defi{consistent with the labeling} if for any two edges $e$ and $f$, the fact that $\vphi(e)=\vphi(f)$ implies that $e$ and $f$ have same label.

An \defi{$X$-covering} is an $X$-morphism between two labeled graphs which is also a covering.
\end{definition}
It follows immediately from the definition that a covering is onto as soon as $\Gamma_2$ is connected.

Every $X$-covering is consistent with the labeling.
Moreover, every covering $\vphi\colon\Gamma_1\to\Gamma_2$ consistent with labeling induces a labeling on $\Gamma_2$ such that $\vphi$ is an $X$-morphism for this labeling.
On the other hand, if $\Gamma_2$ is labeled by $X$, then every covering $\vphi\colon\Gamma_1\to\Gamma_2$ induces a labeling on $\Gamma_1$ such that $\vphi$ is an $X$-covering.

\begin{lemma}\label{StrongCovering}
Let $\Gamma_i\defegal\Schrei{A}{H_i}{X}$ for $i=1,2$ be two Schreier graphs over the same group.
Then there exists an $X$-covering from $\Gamma_1$ to $\Gamma_2$ if and only if $\h_1$ is a subgroup of a conjugate of $\h_2$.
\end{lemma}
\begin{proof}
We have $\group A=\G/\group N$.
Since the correspondence between subgroups of $\group A$ and subgroups of $\G$ containing $\group N$ preserves inclusions and conjugations and induces an isomorphism $\Schrei{G}{M}{X}\iso \Schrei{A}{M/\group N}{X}$, it is sufficient to prove the result for $\G$.

Let $\vphi\colon\Gamma_1\to\Gamma_2$ be an $X$-covering and let $v_0\defegal\vphi(\h_1)$ be the image of the base-vertex of $\Gamma_1$.
Now, the group $\h_1$ is isomorphic (see Lemma \ref{lemmaPathWord}) to the group of closed paths based at the vertex $\h_1$.
This group is itself isomorphic to its image under $\vphi$, which is a subgroup of the group of closed paths based at the vertex $v_0$.
This last group is isomorphic to $g\h g^{-1}$, where $g$ is the label of the path between $v_0$ and $\h_2$.

For the converse, let $\h_1\leq \h=g\h_2g^{-1}$, and $\Gamma\defegal\Schrei{G}{H}{X}$.
It is obvious that $\Gamma$ and $\Gamma_2$ are $X$-isomorphic; indeed, we only change the root.
Hence, to conclude the proof, we only need to show that there exists an $X$-covering from $\Gamma_1$ to $\Gamma$.
Define $\vphi\colon\Gamma_1\to\Gamma$ on the vertices by $\vphi(\h_1g)\defegal\h g$.
We need to check that $\vphi$ is well-defined.
But $\h_1g=\h_1f$ if and only if $gf^{-1}\in \h_1\leq\h$, which implies that $\h g=\h f$.
Now, define $\vphi$ on edges by sending the unique edge leaving $\h_1 g$ and labeled by $x$ to the unique edge leaving $\h g$ and labeled by $x$.
With this definition, all the $\vphi_v$ are bijections and $\vphi$ preserves the labeling.
All that remains to check is that $\vphi$ is a morphism of graphs.
It is immediate from the definition that $\vphi$ preserves initial vertices.
Now, let $e$ be an edge in $\Gamma_1$ with initial vertex $\h_1g$ and label $x$.
The inverse edge $\bar e$ has initial vertex $\h_1gx$ and label $x^{-1}$.
Therefore, $\vphi(e)$ has initial vertex $\h g$ and label $x$, and its inverse has initial vertex $\h gx$ and label $x^{-1}$.
That is $\overline{\vphi(e)}=\vphi(\bar e)$.
\end{proof}
At this point, an obvious but important remark is the fact that if one of the $\h_i$ is normal, the existence of an $X$-covering is equivalent to the fact that $\h_1\leq \h_2$.
This is also true if we ask that the covering preserves roots too.
As an immediate corollary we have:
\begin{proposition}
Let $\group A$ be a group with generating system $X$. Then for any $X^\pm$-labeled graph $\Gamma$, there is an $X$-covering from $\Cayley{A}{X}$ to $\Gamma$ if and only if $\Gamma$ is a Schreier graph over $\group A$.
\end{proposition}
\begin{proof}
We have $\group A=\G/\group N$ with $\group N$ normal and $\Cayley{A}{X}\iso\Schrei{G}{N}{X}$.
There is an $X$-covering if and only if, up to a choice of base point, $\Gamma$ is a Schreier graph of $\G$ for some subgroups $\h$ containing $\group N$. Therefore, $\Gamma$ is a Schreier graph of $\h/\group N$ in $\group A=\G/\group N$.
\end{proof}

\begin{proposition}\label{propCovering}
Let $\Gamma_i\defegal\Schrei{G_i}{H_i}{X_i}$ for $i=1,2$ be two Schreier graphs.
Then there exists a covering from $\Gamma_1$ to $\Gamma_2$ if and only if $\G_1$ and $\G_2$ are of the same degree and $\h_1$ is length-isomorphic to a subgroup of a conjugate of $\h_2$.
\end{proposition}
\begin{proof}
Suppose that there exists a covering.
Then both graphs (and therefore groups) have the same degree.
Moreover, we can pullback by $\vphi$ the labeling of $\Gamma_2$ onto $\Gamma_1$.
Let us denote by $\Gamma=\Schrei{G_2}{H}{X_2}$ the graph obtained in this way.
Apart from the labeling, it is the graph $\Gamma_1$.
Therefore, $\h$ is length-transitive to $\h_1$.
Moreover, due to this new labeling, $\vphi\colon\Gamma\to\Gamma_2$ preserves the labels.
Hence we can use the last lemma to prove that $\h$ is a subgroup of $\h_2$.

The converse is quite obvious.
Let $\h$ the subgroup of $\G_2$ which is length-transitive to $\h_1$.
By Lemma \ref{StrongCovering}, there exists an $X$-covering from $\Gamma$ to $\Gamma_2$.
The graph $\Gamma$ being isomorphic to $\Gamma_1$ (only the labeling changes), we have the desired covering.
\end{proof}

The following lemma is an easy adaptation of a well-known fact about coverings of topological spaces.
\begin{lemma}
For a covering $\vphi\colon\Gamma_1\to\Gamma_2$ and two vertices $w$ and $u$ in the same connected component of $\Gamma_2$, fibers over $w$ and over $u$ have same cardinality.
\end{lemma}
\begin{definition}
If $\vphi\colon\Gamma_1\to\Gamma_2$ is a covering with $\Gamma_2$ connected, the cardinality of the fibers is called the \defi{degree} of the covering.
\end{definition}
\begin{lemma}\label{lemmaDegree}
Let $\h_1\leq\h_2$ be two subgroups of a group $\G$.
Then the index of $\h_1$ in $\h_2$ is equal to the degree of the $X$-covering $\Schrei{G}{\h_1}{X}\to\Schrei{G}{\h_2}{X}$.
\end{lemma}
\begin{proof}
Let us look at the fiber $F$ over the vertex $\h_2$.
It is exactly the set $\enstq{\h_1g}{\h_2g=\h_2}=\enstq{\h_1g}{g\in \h_2}$.
This corresponds to the decomposition of $\h_2$ into right $\h_1$\tiret coset.
\end{proof}

This lemma will allow us to make a link between $X$-covering and quasi-isometries.
\begin{definition}
Let $(M_1,d_1)$ and $(M_2,d_2)$ be two metric spaces.
A \defi{quasi-isometry} from $(M_1,d_1)$ to $(M_2,d_2)$ is a function $f\colon M_1\to M_2$ such that there exists constants $A\leq 1$, $B\leq 0$ and $C\leq 0$ for which
\begin{enumerate}
\item For every two points $x$ and $y$ in $M_1$ we have
\[A^{-1}d_1(x,y)-B\leq d_2(f(x),f(y))\leq Ad_1(x,y)+B.\]
\item For every point $z$ in $M_2$, $d_2(z,f(M_1))\leq C$.
\end{enumerate}

The spaces $M_1$ and $M_2$ are called \defi{quasi-isometric} if there exists a quasi-isometry from $M_1$ to $M_2$.
\end{definition}
The first point means that even if the function $f$ does not necessarily preserve distances, it does not change them too much.
The second point says that $f$ is close to being surjective: every point in $M_2$ is at a bounded distance from the image.
This notion naturally arises in the study of Cayley graphs, since two different finite generating systems for the same group give quasi-isometric Cayley graphs.
Note that every two finite graphs are quasi-isometric.

For a covering $\vphi\colon\Gamma_1\to\Gamma_2$ and a vertex $v\in\Gamma_2$, the \emph{diameter of the fiber of $v$}, $\diam(\vphi^{-1}(v))$, is the maximal distance in $\Gamma_1$ between two preimages of $v$.
For coverings of finite degree, the quasi-isometry of the two graphs follows from one simple condition.
\begin{lemma}\label{lemmaQIso}
Let $\vphi\colon\Gamma_1\to\Gamma_2$ be a covering of finite degree ($\Gamma_2$ is connected).
Suppose that there exists a constant $B$ such that for every vertex $v$ in $\Gamma_2$, the diameter $\diam(\vphi^{-1}(v))$ is at most $B$. Then $\vphi$ is a quasi-isometry with $A=1$ and $C=0$.
\end{lemma}
\begin{proof}
Since $\vphi$ is surjective, we have $C=0$ and we only need to check the first condition in the definition of quasi-isometry.
Moreover, since $\vphi$ maps paths from $v$ to $w$ to paths from $\vphi(v)$ to $\vphi(w)$ we always have
\[d_2(\vphi(v),\vphi(w))\leq d_1(v,w).\]
For the other inequality, take a path $p$ in $\Gamma_2$ that realizes the distance between $\vphi(v)$ and $\vphi(w)$.
This path lifts to a path $\tilde p$ in $\Gamma_1$ from $v$ to $z$ with $z$ in the same fiber as $w$.
This give us the desired inequality:
\[
d_1(v,w)\leq d_1(v,z)+d_1(z,w)\leq d_2(\vphi(v),\vphi(z))+B=d_2(\vphi(v),\vphi(w))+B.
\]
\end{proof}

\begin{lemma}\label{lemmaQISchreier}
Let $\group L\leq\group H\leq\group A$ be two subgroups of $\group A$ such that $\group L$ has finite index in $\h$.
This induces an $X$-covering of finite degree $\vphi\colon\Schrei{A}{L}{X}\to\Schrei{A}{H}{X}$.
Then, for all $l\in \group A$, the supremum 
\[
\sup\enstq{\diam(\vphi^{-1}(\h f))}{f\in\group A,fl^{-1}\in N_{\group A}(\h)\cap N_{\group A}(\group L)}
\]
is finite.
\end{lemma}
\begin{proof}
Firstly, note that $fl^{-1}\in N_{\group A}(\h)\cap N_{\group A}(\group L)$ if and only if $f^{-1}\h f=l^{-1}\h l$ and $f^{-1}\group Lf=l^{-1}\group Ll$.

Let $k$ be the degree of the covering $\vphi$.
Therefore, we have $\h=\group Lg_1\cup\dots\cup \group Lg_k$ for some $g_i\in {\group H}$.
The fiber over the vertex $\group Hl$ is 
\begin{align*}
\enstq{\group Llg}{\h lg= \h l}&=\enstq{\group Llg}{g\in l^{-1}\h l}\\
&=\enstq{\group Lgl}{g\in \h}\\
&=\{\group Lg_1l,\dots, \group Lg_kl\}\\
&=\{\group Llg'_1,\dots, \group Llg'_k\},
\end{align*}
for $g'_i=l^{-1}g_il$.
Therefore, the distance between two vertices in the fiber is at most $2\cdot\max\{\abs{g'_i}\}$.
Indeed, the distance between $\group Llg'_i$ and $\group Llg'_j$ is by the triangular inequality less than or equal to $d(\group Llg'_i,\group L)+d(\group Llg'_j,\group L)$.
On the other side, the fiber over the vertex $\group Hf$ is
\begin{align*}
\enstq{\group Lfg}{\group Hf=\group Hfg}&=\enstq{\group Lfg}{g\in f^{-1}\group Hf}\\
&=\enstq{\group Lfl^{-1}gl}{g\in \group H}\\
&=\{\group Lfg'_1,\dots, Lfg'_k\}.
\end{align*}
Indeed, we have $\group Lfg=\group Lfh$ if and only if $fgh^{-1}f^{-1}$ belongs to $\group L$, if and only if $gh^{-1}$ belongs to $f^{-1}\group Lf=l^{-1}\group Ll$ if and only if $\group Llg=\group Llh$.
Therefore, the $\group Lfg'_i$'s consists of $k$ distinct points, and they are all the fiber over $\h f$.
Hence the distance between two vertices in the fiber over $\group Hf$ is also at most $2\cdot\max\{\abs{g'_i}\}$.
\end{proof}

We are now going to use these two lemmas to prove a classical result about Cayley graphs and small extensions of it.
Recall that a subgroup is \defi{almost normal} if it has only finitely many conjugates and \defi{nearly normal} if it has finite index in its normalizer.
Let us  call an automorphism $\vphi$ of $\Gamma_2$ \emph{compatible with the covering} $\pi\colon\Gamma_1\twoheadrightarrow\Gamma_2$ if there exists an automorphism $\tilde\vphi$ of $\Gamma_1$ such that $\vphi\pi=\pi\tilde\vphi$.
\begin{theorem}\label{thmQIso}
Let $\group A$ be a group with generating system $X$.
Let $\group N\leq \h\leq\group A$ be two subgroups such that $\group N$ has finite index in $\h$.
Then the graphs $\Schrei{A}{N}{X}$ and $\Schrei{A}{H}{X}$ are quasi-isometric if one of the following assumptions holds:
\begin{enumerate}
\item
$\group N$ is almost normal and $\h$ is nearly normal;
\item
$\group N$ and $\h$ are almost normal;\label{pointDeux}
\item
$\Schrei{A}{H}{X}$ is almost transitive by automorphisms compatible with the covering;
\item
$\group N$ is normal, $\h/\group N$ is cyclic of prime order and $\Schrei{A}{H}{X}$ is almost transitive.
\end{enumerate}
\end{theorem}
For the proofs, we have $\group A=\group G/\group L$ and, using the correspondence theorem, it is therefore sufficient to prove the assertion when $\group A=\G$ is a free product of copies of $\Z$ and of $\Z/2\Z$.
See Lemmas \ref{lemmaUn} to \ref{lemmaCinq} for the proofs in this case.
As an immediate corollary of the theorem, we obtain
\begin{proposition}
Let $\group A$ be a group with generating system $X$ and $\group B$ a quotient by a finite normal subgroup.
Then, Cayley graphs $\Cayley{A}{X}$ and $\Cayley{B}{X}$ are quasi-isometric.
\end{proposition}
\begin{proof}
Simply apply part \ref{pointDeux} of the theorem with $\group N=\{1\}$ and $\h$ normal.
\end{proof}

\begin{lemma}\label{lemmaUn}
Let $\group N\leq\group H\leq \group G$ with $\group N$ of finite index in $\h$, $\group N$ almost normal and $\h$ almost normal or nearly normal.
Then, Schreier graphs $\Schrei{G}{N}{X}$ and $\Schrei{G}{H}{X}$ are quasi-isometric.
\end{lemma}
\begin{proof}
By Lemma \ref{lemmaQISchreier},
\[
B_{l}=\sup\enstq{\diam(\vphi^{-1}(\h f))}{f\in\G,f^{-1}\h f=l^{-1}\h l, f^{-1}\group Nf=l^{-1}\group Nl}
\]
is finite.
If $\group H$ is almost normal, and since $\group N$ is almost normal too, there is only finitely many couples of the form $(l^{-1}\h l,l^{-1}\group N l)$.
Therefore, in $B\defegal\sup\enstq{B_l}{l\in \G}$, we only have finitely many different terms and $B$ is finite. We conclude using Lemma \ref{lemmaQIso}.

If $\group H$ is nearly normal, let $M$ denote its normalizer.
We have a sequence of subgroups with finite index inclusion
\[
\group N\hookrightarrow\group H\hookrightarrow\group M
\]
 with $\group N$ almost normal and $\group M$ normal.
 By Lemmas \ref{StrongCovering} and \ref{lemmaDegree}, this is equivalent to the following sequence of coverings of finite degree
 \[
 \Schrei{G}{N}{X}\xrightarrow{\psi}\Schrei{G}{H}{X}\xrightarrow{\vphi}\Schrei{G}{M}{X}.
 \]
 
Therefore, the first part of this lemma gives us
\[
d_{\group N}(v,w)-B\leq d_{\group M}(\vphi\circ \psi(v),\vphi\circ \psi(w))\leq d_{\group H}(\psi(v),\psi(w))\leq d_{\group N}(v,w),
\]
where $d_{\group N}(v,w)$ is the distance in $\Schrei{G}{N}{X}$.
\end{proof}
\begin{lemma}
Let $\group N\leq\group H\leq \G$ with $\group N$ normal and $\group H/\group N$ cyclic of prime order in $\group G/\group N$.
If the graph $\Schrei{G}{H}{X}$ is almost transitive, then it is quasi-isometric to $\Schrei{G}{N}{X}$.
\end{lemma}
\begin{proof}
By Lemma \ref{lemmaQIso}, it is sufficient to find a universal bound $B$ on the distance between vertices in the same fiber.
The graph being almost transitive, there is a finite number of class of vertices under the action of its automorphism group.
It is thus enough to find bounds for fibers over vertices in the same class, the bound $B$ being the maximum over all these bounds.

Choose a vertex $\group Hh$ in $\Schrei{G}{H}{X}$ and let $g$ be an element of minimal length in $h^{-1}\group Hh-\group N$.
Due to the structure of $\group H$, we have that $h^{-1}\group Hh=\group N\langle g\rangle$ and that the fiber over the vertex $\group Hh$ is $\{\group Nh,\group Nhg,\dots,\group Nhg^{p-1}\}$.
Hence the distance between two of its elements is at most $(p-1)\abs g$.
Indeed, the distance between $\group Nhg^i$ and $\group Nhg^j$ is at most $\abs{i-j}\abs g$.
If $\group Hf$ is in the same transitivity class as $\group Hh$, then there is a bijection that preserves lengths between $h^{-1}\group Hh$ and $f^{-1}\group Hf$.
Hence, we can choose $g'$ in $f^{-1}\group Hf-\group N$ of same length as $g$.
Since $\h/\group N$ is cyclic of prime order, we have $f^{-1}\group Hf=\group N\langle g'\rangle$ and, as before, the distance between two elements of the fiber is at most $(p-1)\abs{g'}=(p-1)\abs{g}$.
\end{proof}
\begin{lemma}\label{lemmaCinq}
Let $\group N\leq\group H\leq \G$ be as in Theorem \ref{thmQIso}.
If the graph $\Schrei{G}{H}{X}$ is almost transitive by automorphisms compatible with the covering, then it is quasi-isometric to $\Schrei{G}{N}{X}$.
\end{lemma}
\begin{proof}
Thinking in terms of subgroups, the automorphism $\vphi$ of $\Schrei{G}{H}{X}$ corresponds to a length-preserving isomorphism $\alpha\colon\group H\to f^{-1}\group Hf$.
The compatibility with the covering is then equivalent to $\alpha(\group N)=\group N$.
Fibers over vertices $\group H$ and $\group Hf$ are respectively $\enstq{\group Ng}{g\in \group H}$ and $\enstq{\group Nf\alpha(g)}{g\in \group H}$.

We have
\begin{align*}
\group N=\group Ng&\Leftrightarrow g\in N\\
&\Leftrightarrow \alpha(g)\in\alpha(\group N)=\group N\\
&\Leftrightarrow \group Nf=\group Nf\alpha(g).
\end{align*}
Hence, the application $\group Ng\mapsto\group Nf\alpha(g)$ is a well-defined bijection between the fibers.
Therefore, if the fiber over $\group H$ is given by $\{\group Ng_1,\dots,\group Ng_k\}$, the fiber over $f^{-1}\group Hf$ is $\{\group Nf\alpha(g_1),\dots,\group Nf\alpha(g_k)\}$ with $\abs{g_i}=\abs{\alpha(g_i)}$.

Once again, we conclude using Lemma \ref{lemmaQIso}.
\end{proof}

It is natural to ask if Theorem \ref{thmQIso} can be extended.
It may be possible, but not in full generality.
Indeed, there are examples of subgroups $\group N\leq \group H$ with $\group N$ of finite index in $\group H$ but such that $\Schrei{G}{N}{X}$ and $\Schrei{G}{H}{X}$ are not quasi-isometric.
There are even such examples with $\group N$ or $\h$ normal.

In order to show that some graphs are not quasi-isometric, we will use the notion of ends.
There are different equivalent definitions for the ends of a graph, but for our purpose it is sufficient to know that the number of ends of a locally finite graph $\Gamma$ is the maximal number of infinite connected components of $\Gamma-\Delta$ where $\Delta$ is a finite subgraph (not necessarily connected).
The number of ends is invariant under quasi-isometries.
For a Cayley graph, the number of ends is either $0$ (if and only if the graph is finite), $1$ ($\Z^d$ with $d\geq 2$ for example), $2$ (if and only if the group is virtually $\Z$) or uncountable ($\group F_n$ for $n\geq 2$ for example).

We now exhibit two examples of $\group N\leq \group H$ such that $\group N$ is of finite index in $\h$ but the graphs $\Schrei{G}{N}{X}$ and $\Schrei{G}{H}{X}$ are not quasi-isometric.
Instead of describing the subgroups $\h$ and $\group N$ explicitly, we will simply describe their Schreier graphs and show that there exists an $X$-covering of finite degree between them.
Indeed, by Lemmas \ref{StrongCovering} and \ref{lemmaDegree}, this implies that $\group N$ is a subgroup of finite index of $\group H$.
\begin{example}\label{exampleUn}
The graphs of Figure \ref{GraphesNonQIUn} correspond to subgroups $\group N\leq \group H\leq \langle x,y\rangle=\group F_2$ with $\group N$ of index two in $\h$.
Since $\Schrei{F_2}{H}{\{x,y\}}$ is $X$-transitive, the subgroup $\h$ is normal.
But $\Schrei{F_2}{N}{\{x,y\}}$ has four ends while the graph $\Schrei{F_2}{H}{\{x,y\}}$ has only two ends.
Therefore, the two graphs are not quasi-isometric.
\begin{figure}[!h]
\centering
\begin{tikzpicture}[scale=0.7,>=stealth]
  \SetVertexNoLabel
  \SetGraphUnit{2}
  \tikzset{LabelStyle/.style ={fill=white}}
  \foreach \k in {0,1}{
  	  \Vertex[x=-2*\GraphUnit,y=\k*\GraphUnit]{-2\k}
	  \foreach \l/\t in {-1/-2,0/-1,1/0,2/1}{
  		\Vertex[x=\l*\GraphUnit,y=\k*\GraphUnit]{\l\k}
		\Edge[style={->},label=$x$](\l\k)(\t\k)
  	}
  \node(G) at (-\GraphUnit*2.8,\k*\GraphUnit){};
  \Edge[style={dotted}](-2\k)(G)
  \node(D) at (\GraphUnit*2.8,\k*\GraphUnit){};
  \Edge[style={dotted}](2\k)(D)
  }
  \foreach \l in {-2,-1,1,2}{
  	\Loop[dist=1.5cm,dir=NO,label=$y$,style={thick,->}](\l1)
  	\Loop[dist=1.5cm,dir=SO,label=$y$,style={thick,<-}](\l0)
  }
  \Edges[style={<-,bend left},label=$y$](01,00,01)
  \AddVertexColor{black}{00}
  \AddVertexColor{black}{01}
  \begin{scope}[yshift=-0.5cm]
  \node(H) at (0,0){};
  \node(B) at (0,-1.5){};
  \Edge[style={->>,>=to}](H)(B)
  \end{scope}
  \begin{scope}[yshift=-2.5cm]
  \Vertex[x=-2*\GraphUnit,y=0]{-2}
  \Loop[dist=1.5cm,dir=SO,label=$y$,style={thick,<-}](-2)
  \foreach \l/\t in {-1/-2,0/-1,1/0,2/1}{
  		\Vertex[x=\l*\GraphUnit,y=0]{\l}
		\Edge[style={->},label=$x$](\l)(\t)
		\Loop[dist=1.5cm,dir=SO,label=$y$,style={thick,<-}](\l)
  	}
  \node(G) at (-2.8*\GraphUnit,0){};
  \Edge[style={dotted}](-2)(G)
  \node(D) at (\GraphUnit*2.8,0){};
  \Edge[style={dotted}](2)(D)
  \AddVertexColor{black}{0}
  \end{scope}
\end{tikzpicture}
\caption{An $X$-covering between two Schreier graphs over the free group of rank two. The root of the base graph and its two preimages are marked in black.}
\label{GraphesNonQIUn}
\end{figure}
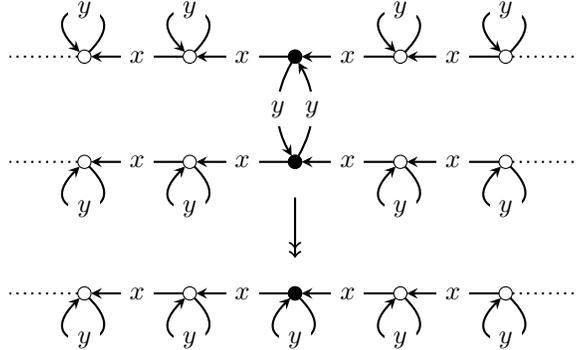
\end{example}
\begin{example}\label{exampleDeux}
The graphs of Figure \ref{GraphesNonQITrois} correspond to subgroups $\group N\leq \group H\leq \presentation{x,a}{a^2}=\Z*\Z/2\Z$ with $\group N$ of index two in $\h$.
The covering is given by the central inversion in respect to the middle point between the two black vertices.
Since $\Schrei{\Z*\Z/2\Z}{N}{\{x,a\}}$ is $X$-transitive, the subgroup $\group N$ is normal and $\Schrei{\Z*\Z/2\Z}{N}{\{x,a\}}\iso\Cayley{(\Z*\Z/2\Z)/\group N}{\{x,a\}}$.
But $\Schrei{F_2}{N}{\{x,y\}}$ has two ends while $\Schrei{F_2}{H}{\{x,y\}}$ has only one end.
Therefore, the two graphs are not quasi-isometric.

Graphs of Figure \ref{GraphesNonQIDeux} shows a similar example, with $\group N\leq \group H\leq \langle x,y\rangle=\group F_2$.
Since $\Schrei{F_2}{N}{\{x,y\}}$ is almost $X$-transitive, the subgroup $\group N$ is this time almost normal.
In fact, $\group N$ has only two conjugates: itself (corresponding to the black vertex in Figure \ref{GraphesNonQIDeux}) and $y^{-1}Ny$ (corresponding to the dark gray vertex in Figure \ref{GraphesNonQIDeux}).
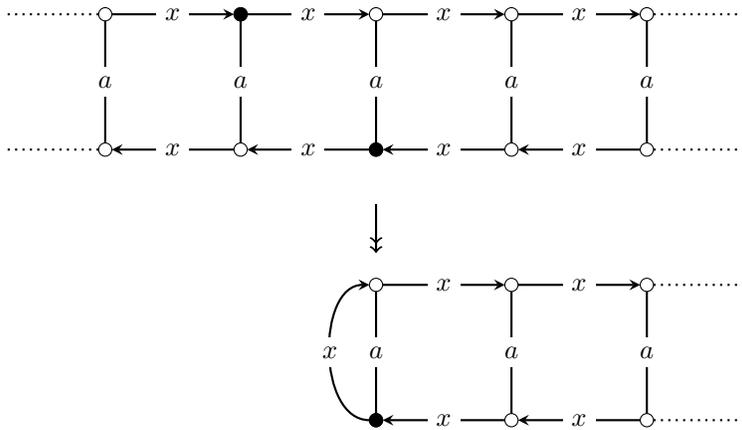
\begin{figure}[!h]
\centering
\begin{tikzpicture}[scale=0.6,>=stealth]
  \SetVertexNoLabel
  \SetGraphUnit{3}
  \tikzset{LabelStyle/.style ={fill=white}}
  %sommets
  \foreach \x in {-2,...,2}{
  	\foreach \y in {0,1}{
			\coordinate(\x\y) at (\x*\GraphUnit,\y*\GraphUnit){};
			\Vertex[Node]{\x\y}
 	 }
  }
%  %arêtes en bas au fond
  \foreach \x/\t in {-1/-2,0/-1,1/0,2/1}{
  	\Edge[label=$x$,style={->}](\x0)(\t0)
  	\Edge[label=$x$,style={->}](\t1)(\x1)
  }
  %arêtes y
  \foreach \x in {-2,...,2}{
  	 \Edges[label=$a$](\x0,\x1)
  }
  \foreach \y in {0,1}{
	\node(G) at (-2.8*\GraphUnit,\y*\GraphUnit){};
	\Edge[style={dotted}](-2\y)(G)
	\node(D) at (2.8*\GraphUnit,\y*\GraphUnit){};
	\Edge[style={dotted}](2\y)(D)
  }
  \AddVertexColor{black}{-11}
  \AddVertexColor{black}{00}
  \begin{scope}[yshift=-1cm]
  \node(H) at (0,0){};
  \node(B) at (0,-1.5){};
  \Edge[style={->>,>=to}](H)(B)
  \end{scope}
  \begin{scope}[yshift=-6cm]
  %sommets
  \foreach \x in {0,...,2}{
  	\foreach \y in {0,1}{
			\coordinate(\x\y) at (\x*\GraphUnit,\y*\GraphUnit){};
			\Vertex[Node]{\x\y}
 	 }
  }
  %arêtes y
  \foreach \x in {0,...,2}{
	 \Edges[label=$a$](\x0,\x1)
  }
  %arêtes restantes
  \foreach \x/\t in {1/0,2/1}{
			\Edge[label=$x$,style={->}](\x0)(\t0)
			\Edge[label=$x$,style={->}](\t1)(\x1)
  }
  \foreach \y in {0,1}{
	\node(D) at (2.8*\GraphUnit,\y*\GraphUnit){};
	\Edge[style={dotted}](2\y)(D)
  }
  \AddVertexColor{black}{00}
  \Edge[label=$x$,style={->,bend left=90}](00)(01)
  \end{scope}
\end{tikzpicture}
\caption{An $X$-covering between two Schreier graphs over $\presentation{x,a}{a^2}$. The root of the base graph and its two preimages are marked in black.}
\label{GraphesNonQITrois}
\end{figure}
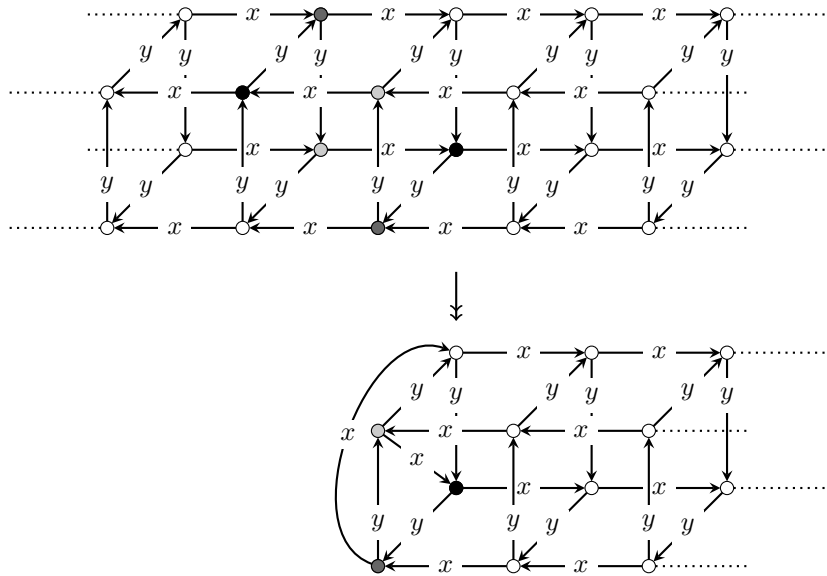
\begin{figure}[!h]
\centering
\begin{tikzpicture}[scale=0.6,>=stealth]
  \SetVertexNoLabel
  \SetGraphUnit{3}
  \tikzset{LabelStyle/.style ={fill=white}}
  %sommets
  \foreach \x in {-2,...,2}{
  	\foreach \y in {0,1}{
  		\foreach \z in {0,1}{
			\coordinate(\x\y\z) at (\x*\GraphUnit,\y*\GraphUnit,1.5*\z*\GraphUnit){};
			\Vertex[Node]{\x\y\z}
		}
 	 }
  }
  %arêtes en bas au fond
  \foreach \x/\t in {-1/-2,0/-1,1/0,2/1}{
  	\Edge[label=$x$,style={->}](\t00)(\x00)
  }
  %arêtes y
  \foreach \x in {-2,...,2}{
  	 \Edges[label=$y$,style={->}](\x00,\x01)
	 \Edges[labelstyle={pos=0.3},label=$y$,style={->}](\x01,\x11)
	 \Edges[label=$y$,style={->}](\x11,\x10)
	 \Edges[labelstyle={pos=0.3},label=$y$,style={->}](\x10,\x00)
  }
  %arêtes restantes
  \foreach \x/\t in {-1/-2,0/-1,1/0,2/1}{
		\foreach \y in {0,1}{
			\Edge[label=$x$,style={->}](\x\y1)(\t\y1)
		}
		\Edge[label=$x$,style={->}](\t10)(\x10)
  }
  \foreach \y in {0,1}{
  	\foreach \z in {0,1}{
	\node(G) at (-2.8*\GraphUnit,\y*\GraphUnit,1.5*\z*\GraphUnit){};
	\Edge[style={dotted}](-2\y\z)(G)
	\node(D) at (2.8*\GraphUnit,\y*\GraphUnit,1.5*\z*\GraphUnit){};
	\Edge[style={dotted}](2\y\z)(D)
	}
  }
  \AddVertexColor{black}{-111}
  \AddVertexColor{black}{000}
  \AddVertexColor{black,opacity=0.6}{-110}
  \AddVertexColor{black,opacity=0.6}{001}
  \AddVertexColor{black,opacity=0.2}{-100}
  \AddVertexColor{black,opacity=0.2}{011}
  \begin{scope}[yshift=-2.5cm]
  \node(H) at (0,0){};
  \node(B) at (0,-1.5){};
  \Edge[style={->>,>=to}](H)(B)
  \end{scope}
  \begin{scope}[yshift=-7.5cm]
  %sommets
  \foreach \x in {0,...,2}{
  	\foreach \y in {0,1}{
  		\foreach \z in {0,1}{
			\coordinate(\x\y\z) at (\x*\GraphUnit,\y*\GraphUnit,1.5*\z*\GraphUnit){};
			\Vertex[Node]{\x\y\z}
		}
 	 }
  }
  %arêtes en bas au fond
  \foreach \x/\t in {1/0,2/1}{
  	\Edge[label=$x$,style={->}](\t00)(\x00)
  }
  %arêtes y
  \foreach \x in {0,...,2}{
  	 \Edges[label=$y$,style={->}](\x00,\x01)
	 \Edges[labelstyle={pos=0.3},label=$y$,style={->}](\x01,\x11)
	 \Edges[label=$y$,style={->}](\x11,\x10)
	 \Edges[labelstyle={pos=0.3},label=$y$,style={->}](\x10,\x00)
  }
  %arêtes restantes
  \foreach \x/\t in {1/0,2/1}{
		\foreach \y in {0,1}{
			\Edge[label=$x$,style={->}](\x\y1)(\t\y1)
		}
		\Edge[label=$x$,style={->}](\t10)(\x10)
  }
  \foreach \y in {0,1}{
  	\foreach \z in {0,1}{
	\node(D) at (2.8*\GraphUnit,\y*\GraphUnit,1.5*\z*\GraphUnit){};
	\Edge[style={dotted}](2\y\z)(D)
	}
  }
  \AddVertexColor{black}{000}
  \AddVertexColor{black,opacity=0.6}{001}
  \AddVertexColor{black,opacity=0.2}{011}
  \Edge[label=$x$,style={->}](011)(000)
  \Edge[label=$x$,style={->,bend left=90}](001)(010)
  \end{scope}
\end{tikzpicture}
\caption{An $X$-covering between two Schreier graphs over the free group of rank two. The root of the base graph and its two preimages are marked in black.}
\label{GraphesNonQIDeux}
\end{figure}
\end{example}

Since people are usually interested in Schreier graphs over free groups and Cayley graphs without loops or multiple edges, it is natural to ask the following question.
\begin{question}
Is it possible to find $\group N\leq \group H\leq \group F_n$ such that $\group N$ has finite index in $\h$, the Schreier graphs $\Schrei{F_n}{N}{X}$ and $\Schrei{F_n}{H}{X}$ are both simple (without loop or multiple edges) and non quasi-isometric and such that at least one of $\group N$ or $\h$ is normal ?
\end{question}
The second part of Example \ref{exampleDeux} shows that this is possible if we replace normality by almost normality and Example \ref{exampleUn}  shows that this is possible if we do not ask the graphs to be simple.

\section{Application to groups}\label{sectionGroups}
We have seen that for a subgroup $\h$ of $\G$, length-transitivity is a weak version of normality.
More generally, for any group $\group A$ and subgroup $\h$, asking for the transitivity of $\Schrei{A}{H}{X}$ is a weak version of the normality.
Since the normality does not depend on the generating system, it is natural to ask if the same is true for the transitivity of the Schreier graph.
It turns out that this is not the case.
We will prove in the next proposition that for all subgroups, there exists a (big) generating system such that the corresponding Schreier graph is transitive.
Moreover, even if we restrict ourself to ``reasonable'' generating systems, the only subgroups such that all Schreier graphs are transitive are the normal subgroups (Proposition \ref{TransNotInvariant}).

For a group $\group A$, denote by $d(\group A)$ the number of elements of order $2$ plus half of the number of elements of order at least $3$.
\begin{proposition}\label{PropAllTransitive}
Let $\group A$ be a group and $\h$ any subgroup.
Then there exists a generating system $X$ of size $d(\group A)$ such that $\Schrei{A}{H}{X}$ is transitive.
\end{proposition}
\begin{proof}
Any group $\group A$ can be decomposed as $\group A=\{1\}\sqcup S\sqcup T\sqcup T^-$, where $S$ consists of elements of order $2$, $T$ of half of elements of order at least $3$ and $T^-$ of inverses of elements in $T$.
Let $X\defegal S\cup T$.
Then, $X^\pm=\group A\setminus\{1\}$ and $\abs X=d(\group A)$.

Let $\group A=\bigsqcup Hg_i$ be the decomposition into right cosets.
For each $i$ and $j$, there is an edge labeled by $g$ from $Hg_i$ to $Hg_j$ if and only if $g\in g_i^{-1}Hg_j$.
Since $\abs{g_i^{-1}Hg_j}=\abs H$, for any two vertices in $\schrei(\group A,\h,\group A)$, there is exactly $\abs H$ edges going from $v$ to $w$ and this graph is transitive (it is a thick complete graph).
The edges labeled by $1$ being always loop, the graph $\Schrei{A}{H}{X}$ is also transitive.
\end{proof}

We now prove that even if we restrict ourself to generating systems of size at most $\rank(\group A)+1$, the fact that $\Schrei{A}{H}{X}$ is transitive does depend on $X$ if $\h$ is not normal.
\begin{lemma}
Let $\group A$ be a group (not necessarily finitely generated) and let $\h$ be a proper subgroup.
Then there exists a generating system $X$ of $\group A$ such that $X\cap \h$ is empty and $\abs X= \rank(\group A)$.
\end{lemma}
\begin{proof}
Let $X$ be a generating system of $\G$ such that $\abs X=\rank(\group A)$.
If $X\cap \h$ is empty, the assertion is true.
Therefore, we can suppose that $X\cap \h$ is not empty.
Since $\h$ is a proper subgroup, we also have $X\cap \h\neq X$.
Thus, we can order the elements of $X$ and find an $x_0\in X$ such that $x\in X$ belongs to $\h$ if and only if $x< x_0$.
Now, take $Y\defegal\enstq{x}{x\in X,x\geq x_0}\sqcup\enstq{xx_0}{x\in X, x<x_0}$.
This is trivially a generating system of the same cardinality as $X$, and $Y\cap \h$ is empty.
Indeed, if $x\geq x_0$ then $x\notin\h$.
But if $x<x_0$ and $xx_0$ belongs to $\h$, we have $x_0=x^{-1}xx_0\in\h$, which is absurd.
\end{proof}
\begin{proposition}\label{TransNotInvariant}
Let $\group A$ be a group and $\h$ a subgroup of $\group A$.
If, for all generating systems $X$ of size at most $\rank(\group A)+1$, the graph $\Schrei{A}{H}{X}$ is transitive, then $\h$ is a normal subgroup of $\group A$.
\end{proposition}
\begin{proof}
If $\h=\group A$, there is nothing to prove.
Therefore, we can suppose that $\h$ is a proper subgroup and find, by the preceding lemma, a generating system $X$ such that $\abs X=\rank(\group A)$ and $X\cap\h=\emptyset$.
The Schreier graph $\Schrei{A}{H}{X}$ is transitive by assumption and does not have loops since $X\cap\h=\emptyset$.
For any $h\in \h$, let $X_h\defegal X\sqcup \{h\}$; a generating system of size $\rank(\group A)+1$.
The graph $\Schrei{A}{H}{X_h}$ is transitive and has a unique loop (labeled by $h$) at the vertex $\h$.
Therefore, for all $g\in \group A$, the vertex $\h g$ as a unique loop.
The label of this loop is $h$ since the graph $\Schrei{A}{H}{X}$ has no loops at the vertex $\h g$.
But this implies that for all $g\in \group A$, $\h gh=\h g$.
Therefore, for all $h\in\h$ and $g\in\group A$, $ghg^{-1}$ belongs to $\h$ and we have just proven that $\h$ is normal.
\end{proof}

In the following, we will only take in account locally finite graphs and finite generating systems of groups.
This is justified by the fact that if $\group A$ is not finitely generated, then $\abs{\group A}=d(\group A)=\rank(\group A)$, but has other important consequences for the study of Schreier graphs.
For example, every connected, locally finite transitive graph is a Schreier graph, see Section \ref{sectionSchreier}.

Due to Proposition \ref{TransNotInvariant}, we know that the transitivity of $\Schrei{A}{H}{X}$ does not only depends on $\h$, but on $X$ too if $\h$ is non-normal.
This and Proposition \ref{PropAllTransitive} motivate the following definition.
\begin{definition}
A finitely generated $\group A$ is \defi{strongly simple} if for any generating system $X$ of size at most $\rank(\group A)+1$, and any proper subgroup $\{1\}<\h<\group A$, the graph $\Schrei{A}{H}{X}$ is not transitive.
\end{definition}

It is immediate that strong simplicity implies simplicity and that cyclic groups of prime order $\group C_p$ are strongly simple.
Indeed, such groups do not have proper subgroups.

Proposition \ref{TarskiStronglySimple} shows the existence of infinite strongly simple groups, proving that the class of strong simple groups is not reduced to cyclic groups.
On the other side, the following proposition show that there exists (finite) simple groups which are not strongly simple.

\begin{proposition}\label{ExAn}
For odd $n\geq 7$, let $\h_n$ be the subgroup of $\group A_n$ consisting of elements fixing $n$ and let $a_n\defegal (1,3,4,5,\dots n,2)$ and $b_n\defegal (2,4,6,\dots n-1,1,n,n-2,\dots,5,3)$.
Then $\{a_n,b_n\}$ generates $\group A_n$ and the graph $\Schrei{A_n}{H_n}{\{a_n, b_n\}}$ is transitive. 

In particular, $A_n$ is simple but not strongly simple.
\end{proposition}
\begin{proof}
We have that $\h_n$ is isomorphic to $\group A_{n-1}$ and has index $n$.
Moreover, the right cosets for $\h_n$ depend only on the preimage of $n$.
Therefore, the action of $\group A_n$ on $\group A_n/\h_n$ is isomorphic to the action of $\group A_n$ on $\{1,\dots, n\}$.
It is then easy to see that $\Schrei{A_n}{H_n}{\{a_n, b_n\}}$ is isomorphic to the circulent graph $C_n^{1,2}$ (see Figure \ref{AnNotStronglyTrans} for an example): each vertex $i$ has $4$ neighborhood: $i\pm 1, i\pm2$.
Such a graph is obviously transitive.

\begin{figure}
\centering
\begin{tikzpicture}[>=stealth,a/.style={->,densely dotted},b/.style={->}]
  \SetVertexNoLabel
  \SetGraphUnit{2}
  \Vertices{circle}{1,2,3,4,5,6,7}
  \Edges[style={a,bend right=10}](3,4,5,6,7)
  \Edges[style={a,bend left=10},label=$a_7$](2,1)
  \Edges[style={a}](1,3)
  \Edges[style={a}](7,2)
  \Edges[style={b}](7,5,3)
  \Edges[style={b}](2,4,6,1)
  \Edges[style={b,bend left=10}](1,7)
  \Edges[style={b,bend left=10}, label=$b_7$](3,2)
  \AddVertexColor{black}{1}
\end{tikzpicture}
\caption{The graph $\Schrei{A_7}{H_7}{\{a_7, b_7\}}$.The root (the vertex $\h_n$) is marked in black.}
\label{AnNotStronglyTrans}
\end{figure}
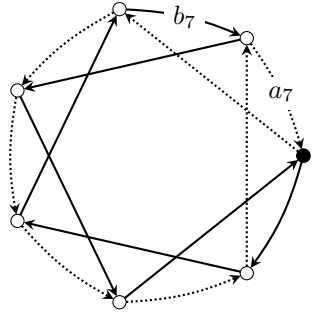

In order to finish the proof, we need to show that $a_n$ and $b_n$ generate $\group A_n$. For $n\geq 7$, a direct computation gives $b_na_n^{-2}b_n^{-1}a_n^2=(4,3,n-1)$.
We conclude using the fact that $(4,3,n-1)$ and $a_n$ generates $\group A_n$ for odd $n\geq 5$ (see \cite{alternating}).

\end{proof}
The above proof does not work in the case where $n=5$ or $n$ is even. For $n=5$, the graph is still transitive, but $b_5^2=a_5$ and therefore $\{a_5, b_5\}$ does not generate $\group A_5$.
A careful check shows that if $\Schrei{A_5}{H_5}{X}$ is transitive, then $X$ has at least $3$ elements and $3$ is possible (take $a_5$, $b_5$ and $(1,2,3,4,5)$).
For $n=6$, we even have that if $\Schrei{A_5}{H_5}{X}$ is transitive, then $X$ has at least $4$ elements.
More generally, for even $n$, let $c_i=(1,2,3,\dots ,\hat i,\dots,n)$ (the cycle $(1,2,\dots, n)$ without $i$).
Then the $c_i$'s generate $\group A_n$ and $\Schrei{A_n}{H_n}{\{c_i\}}$ is transitive. In this case, we have a generating set of size $n$, which is small if we compare it to $d(\group A_n)> \frac{n!}4$ but big if we compare it to $\rank(\group A_n)=2$.

We now turn our attention on the cyclic subgroups of prime order.
\begin{proposition}
Let $\group A$ be finitely generated group, $X$ a finite generating system, and $\group K$ a cyclic subgroup of prime order.
Then, in the graph $\Schrei{A}{K}{X}$, each orbit is a finite union of $X$-orbits.
\end{proposition}
\begin{proof}
We have $\group A=\group G/\group N$, with $\group G$ finitely generated.
The subgroup $\group K$ corresponds to a subgroup $\h$ of $\group G$ containing $\group N$.
For any vertex $\h g$ in the graph $\Schrei{G}{H}{X}\iso\Schrei{A}{K}{X}$, its orbit $[\h g]$ is the set of all vertices $\h f$ such that there exists an automorphism mapping $\h g$ to $\h f$.
Since $\group K$ is cyclic of prime order, we have $g^{-1}\h g=\group N\langle h\rangle$ for every $h$ in $g^{-1}\h g-\group N$.
We can choose $h$ to have minimal length among elements of $g^{-1}\h g-\group N$.
For any vertex $\h f$ in $[\h g]$, there exists a bijection from $g^{-1}\h g$ to $f^{-1}\h f$ which preserves lengths.
Hence, there exists $h'$ in $f^{-1}\h f-\group N$ which has same length as $h$; and we have $f^{-1}\h f=\group N\langle h'\rangle$.
Since $\group G$ is finitely generated, its set of elements of length $\abs h$ is finite.
Thus, we have that the set of subgroups $\enstq{f^{-1}\h f}{\h f\in[\h g]}$ is finite.
We conclude the proof using the fact that the $X$-orbit of $\h g$ consists exactly of vertices $\h f$ such that $f^{-1}\h f=g^{-1}\h g$.
\end{proof}
\begin{corollary}
Let $\group A$ be a infinite simple group, $X$ a finite generating system, $\group K$ a proper non-trivial subgroup and $\Gamma\defegal\Schrei{A}{K}{X}$.
Then $\Aut_X(\Gamma)$ has an infinite number of orbits.
Moreover, if $\group K$ is cyclic of prime order, then $\Aut(\Gamma)$ has an infinite number of orbits and therefore, $\Gamma$ is not almost transitive.
\end{corollary}
\begin{proof}
Since $\group A$ is infinite simple, it does not have any finite index subgroups.
Therefore, the number of $X$-orbits, which is $[\group A: N_\group A(\group K)]$, is infinite.
The last proposition implies that if $\group K$ is cyclic of prime order, then the number of orbits is also infinite.
\end{proof}

Recall that a \defi{Tarski monster $\group T_p$} is an infinite group such that every proper subgroup is isomorphic to a cyclic group of order $p$, for $p$ a fixed prime.
It follows directly from the definition that every such group has rank $2$ and is simple.
Ol'shanskii proved in \cite{Ol} that there exists uncountably many non-isomorphic Tarski monsters for every $p$ greater that $10^{75}$.
\begin{proposition}\label{TarskiStronglySimple}
Let $\group T_p$ be a Tarski monster, $X$ a finite generating system, $\group K$ a proper non-trivial subgroup and $\Gamma$ the corresponding Schreier graph.
Then $\Aut_X(\Gamma)=\{1\}$ and each orbit of $\Aut(\Gamma)$ is finite.
In particular, $\Gamma$ is not almost transitive
 and Tarski monsters are strongly simple.
\end{proposition}
\begin{proof}
Due to the particular structure of subgroups in $\group T_p$, we have $N_{\group T_p}(\group K)=\group K$, which implies that all $X$-orbits are singletons.
Thus, $\Aut_X(\Gamma)=\{1\}$.
Since $\group K$ is cyclic of prime order, each orbit is a finite union of $X$-orbits and therefore finite.
\end{proof}
The existence of strongly simple infinite groups partially answer to a question of Benjamini and Duminil-Copin:
\begin{question}
Does there exists a constant $M$ such that every infinite transitive (Cayley) graph $\Gamma$, not quasi isometric to $\mathbf Z$, covers an infinite transitive graph $\Delta$ of girth at most $M$ and such that $\Delta$ is non quasi-isometric to $\Gamma$ ?
\end{question}
The original motivation for this question was a conjecture about the connective constant of transitive graphs. This conjecture was solved by Grimmet and Li in \cite{ConnConst}.
But the question of Benjamini and Duminil-Copin remains and the following weak form is still an open problem:
\begin{conjecture}\label{conjBenjaminiFaible}
Every infinite Caley graph $\Gamma$, not quasi isometric to $\mathbf Z$, covers an infinite transitive graph $\Delta\not\iso\Gamma$.
\end{conjecture}
If we ask for $X$-coverings, the conjecture is false, with Cayley graph of infinite strongly simple groups as counter-examples.
This is a reason to believe that the conjecture itself is false, with Cayley graph of strongly simple groups as possible counter-examples.

\end{document}